\documentclass[12pt]{amsart}
\usepackage{amssymb}
\usepackage{amsmath, amsthm, amscd, amsfonts, amssymb, graphicx,mathrsfs, color}
\usepackage[colorlinks=true, linkcolor=blue,urlcolor=blue]{hyperref}
\input xy \xyoption{all}
\usepackage[all]{xy}
\usepackage{layout}
\usepackage{tikz}
\usepackage{bm}
\usepackage{geometry}
\usepackage{euscript}
\usepackage{enumerate}

\newtheorem{thm}{Theorem}[section]
\newtheorem{cor}[thm]{Corollary}
\newtheorem{lem}[thm]{Lemma}
\newtheorem{prop}[thm]{Proposition}
\theoremstyle{definition}
\newtheorem{defn}[thm]{Definition}
\newtheorem{ex}[thm]{Example}
\theoremstyle{remark}
\newtheorem{rem}[thm]{Remark}

\numberwithin{equation}{section}

\begin{document}

\author[Ruhollah Barati]{Ruhollah Barati}
\email{baratiroohollah@gmail.com}

\title{On some extensions of strongly unit nil-clean rings}

\begin{abstract}
An element $x \in R$ is considered (strongly) nil-clean if it can be expressed as the sum of an idempotent $e \in R$ and a nilpotent $b \in R$ (where $eb = be$). If for any $x \in R$, there exists a unit $u \in R$ such that $ux$ is (strongly) nil-clean, then $R$ is called a (strongly) unit nil-clean ring. It is worth noting that any unit-regular ring is strongly unit nil-clean. 
In this note, we provide a characterization of the unit regularity of a group ring, along with an additional condition. We also fully characterize the unit-regularity of the group ring $\mathbb{Z}_nG$ for every $n > 1$. Additionally, we discuss strongly unit nil-cleanness in the context of Morita contexts, matrix rings, and group rings.    
\end{abstract}

\maketitle

\section{Introduction}
Throughout, all rings are associative with an identity. Almost all notions
and notations are standard being in agreement with the well-known book \cite{lam2013first}.
We denote the Jacobson radical, the lower nil-radical, the Levitzki radical, and the set nilpotent elements of a ring $R$,  by the symbols $J(R)$, $\mathrm{Nil}_*(R)$, $L(R)$, and $\mathrm{Nil}(R)$, respectively. Moreover, note that the following inclusions hold:

\[\mathrm{Nil}_*(R)\subseteq L(R) \subseteq \mathrm{Nil}(R)\cap J(R).\]

\medskip

Recall that a ring $R$ is said to be {\it $2$-primal} if $\mathrm{Nil}_*(R)=\mathrm{Nil}(R)$. For example, every commutative or reduced ring is $2$-primal. A ring $R$ is called {\it weakly $2$-primal} if the equality $\mathrm{Nil}(R)=L(R)$ holds. It is clear that every $2$-primal ring is weakly $2$-primal. However, the converse implication fails as there exist weakly $2$-primal rings that are not $2$-primal (see, for instance, \cite[Example 2.2]{Marks}). Also, a ring $R$ is called $\mathrm{NI}$ if the set $\mathrm{Nil}(R)$ forms an ideal. The notion of clean rings was initially introduced by Nicholson \cite{Nicholson} in 1977 while working on the exchange property of rings. Specifically, a ring $R$ is called clean if each of its element is a sum of an idempotent and a unit. Since the introduction of clean rings, many varieties of them like (strongly) nil-clean rings have been developed and researched. Diesl \cite{Diesl2013nil} was the first to introduce the notion of nil-clean rings and to establish the (strongly) nil-clean theory. A ring $R$ is called (strongly) nil-clean if each of its element is a sum of an idempotent and a nilpotent (that commutes with each other). It is straightforward to verify a nil-clean ring is clean. There are two main varieties of clean rings: periodic and strongly $\pi$-regular rings. The notion of periodic rings was introduced by Chacron \cite{Chacron} in his seminal article. An element $a \in R$ is considered periodic if $a^m = a^n$ for two distinct natural numbers $m$, $n$. A ring $R$ is called periodic if all its elements are periodic. Regular rings and their varieties are among the most important and fascinating rings that have been studied by many researchers over the years. 

1. An element $a \in R$ is said to be von neumann regular or regular for short if there exists an $x \in R$ such that $a = axa$. If $x \in R$ is invertible, then $a$ is called unit-regular.

2. An element $a \in R$ is said to be strongly regular if $a \in a^2R \cap Ra^2$.

 A ring $R$ is called (unit-)regular if all its elements are (unit-)regular. A ring $R$ is called (strongly) $\pi$-regular if for any $a \in R$, there exist a natural number $n$ such that $a^n$ is (strongly) regular. Inspired \cite{Diesl2013nil}, Chen et al. \cite{SheibaniSUN} introduced the strongly unit nil-clean rings as a generalization of strongly nil-clean rings. Specifically, a ring R is said to be strongly unit nil-clean if, for any $a \in R$, there is a unit $u \in R$ such that $ux$ is strongly nil-clean. They proved that a strongly $\pi$-regular ring is strongly unit nil-clean. It is clear that a unit-regular ring is strongly unit nil-clean. Also, we know that a right (left) perfect ring is strongly $\pi$-regular. According to this, the class of strongly unit nil-clean rings is relatively large. Therefore, we will have the following chart.
 
 
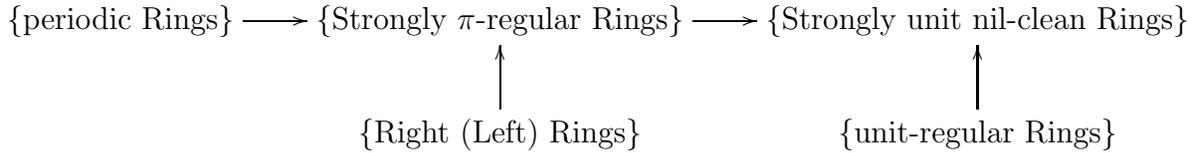
\[\xymatrix{
  \{\text{periodic Rings}\} \ar[r] & \{\text{Strongly $\pi$-regular Rings}\} \ar[r] & \{\text{Strongly unit nil-clean Rings}\} \\ 
 &\{\text{Right (Left) Rings}\} \ar[u] & \{\text{unit-regular Rings}\} \ar[u] &
 }\] 
 
The regularity of group rings was completely characterized by Connell \cite{Connell} and their perfectness by Woods \cite{Woods} as follows. 

\begin{thm}[{\cite[Theorem 3]{Connell}}]\label{Theorem0.1}
The group ring $RG$ is regular if and only if
\begin{enumerate}
\item[(1)] 
$R$ is regular,
\item [(2)]
$G$ is locally finite and,
\item [(3)]
each $n \in o(G)$ is a unit in $R$, where $o(G)$ denotes the set of orders
of all finite subgroups of G.
\end{enumerate}
\end{thm} 

\begin{thm}[{\cite[Theorem]{Woods}}]\label{Theorem0.2}
The group ring $RG$ is perfect if and only if $R$ is perfect and
$G$ is finite.
\end{thm}
\newpage
\noindent Inspired by the above theorems, we pose two following questions.

\textbf{Question 1:} Does Theorem \ref{Theorem0.1} hold in unit-regularity situation?

\textbf{Question 2:} Assuming $R$ is a strongly unit nil-clean ring and $G$ is a locally finite group, is the group ring $RG$ strongly unit nil-clean? In the special case, let $R$ be a right (left) perfect or unit-regular ring and $G$ be a locally finite group, is the group ring $RG$ strongly unit nil-clean?  

The main purpose of this article is to provide answers to the questions mentioned above. In the unit-regularity situation, we demonstrate that Theorem \ref{Theorem0.1} holds with an additional requirement and also we provide partial answers to the question 2.

\section{Basic Results}
In this section, we will prove two important results that will be needed in the following sections.

\begin{defn}\label{Definition0.1}
A ring $R$ is called (strongly) unit nil-clean if for any $a \in R$, there exists a unit $u \in R$, such that $ua$ is (strongly) nil-clean.
\end{defn}

\begin{lem} \label{Lemma1.4}
Let $R$ be a ring such that $R=S+K$, where $S$ is an unital subring of $R$, and $K$ is a nil ideal of $R$. Then $R$ is (strongly) unit nil-clean if and only if $S$ is (strongly) unit nil-clean.
\end{lem}

\begin{proof}
It is clear that $S \cap K \subseteq K$ is a nil ideal of $S$. According to the second isomorphism theorem of rings, $\frac{R}{K}=\frac{S+K}{K} \cong \frac{S}{S \cap K}$. Now Using \cite[Lemma 3.1]{Sheibani} (\cite[Lemma 2.1]{SheibaniSUN}) twice, the proof will be completed.
\end{proof}

A graded ring is a ring whose underlying additive group is a direct sum of abelian groups $R_i$ such that $R_{i}R_{j}\subseteq R_{i+j}$ for all nonnegative numbers $i$ and $j$. The index set is usually the set of nonnegative integers or the set of integers, but can be any monoid. Generally, the index set of a graded ring is assumed to be the set of nonnegative integers, unless otherwise explicitly specified. This is the case in this article.
A nonzero element of $R_{n}$ is said to be homogeneous of degree $n$. By the definition of the direct sum, every nonzero element $a$ of $R$ can be uniquely written as a sum $a=a_{0}+a_{1}+\cdots +a_{n}$, where each $a_{i}$ is either 0 or homogeneous of degree $i$. The nonzero $a_{i}$'s are the homogeneous components of $a$.

\begin{thm}\label{Theorem1.5}
Let $R=\oplus_{i \in \mathbb{N}} R_i$ be a graded ring and $n \in \mathbb{N}$. Suppose $R_i=0$ for $i>n$. Then $R$ is (strongly) unit nil-clean if and only if $R_0$ is (strongly) unit nil-clean.  
\end{thm}

\begin{proof}
Set $T=\oplus_{i=1}^n R_i$. So $R=R_0 \oplus T$. It is easy to verify that $R_0$ is an unital  subring of $R$ and $T$ is a nil ideal. In fact, if $x=\sum_{i=1}^{n} r_i \in T$ then $x^{n+1}=\sum_{i_1,i_2, \cdots,i_{n+1}=1}^{n}
r_{i_1}r_{i_2}\cdots r_{i_{n+1}} \in R_{i_1}R_{i_2} \cdots R_{i_{n+1}} \subseteq R_{i_1+i_2+\cdots+i_{n+1}}$. As $i_1+i_2+\cdots+i_{n+1} \geq n+1$ and $R_{i_1+i_2+\cdots+i_{n+1}}=0$, $x^{n+1}=0$.
  Now Lemma \ref{Lemma1.4} gives the desired result, as respected. 
\end{proof}

\section{Mortita Context and Matrix Rings}
In this section, we present some results related to the strongly unit nil-cleanness of matrix rings and Morita context rings. We begin with the following simple technicality.

\begin{prop}\label{proposition1.1}
The ring $R$ is unit-regular if and only if the full matrix ring $\mathbb{M}_n(R)$ is unit-regular for any $n \in \mathbb{N}$. 
\end{prop}

\begin{proof}
($\Rightarrow$) \quad Let $R$ be an unit-regular ring. Consider the right $R$-module $R_R$. Set $A={(R_R)}^n$. Then $A$ is a free right $R$-module with a finite basis. So $A$ is a finitely generated projective right $R$-module. Now $End_R(A) \cong \mathbb{M}_n(R)$ is an unit-regular ring by \cite[Corollary 4.7]{Goodearl}. 

($\Leftarrow$) \quad It follows from \cite[Theorem 7]{hartwig}.
\end{proof}

\begin{lem}\label{Lemma1.2}
Let $D$ be a domain. The full matrix ring $\mathbb{M}_n(D)$ is strongly unit nil-clean if and only if $D$ is a division ring.
\end{lem}

\begin{proof}
Take an arbitrary element $0 \neq a \in D$. 
For $A=aI_n$, it follows from  \cite[Theorem 2.2]{SheibaniSUN} that there is an integer $t>1$ and $U \in GL_n(D)$ such that ${(AU-{(AU)}^2)}^t=0$. The last equation implies that ${(aU)}^t{(I_n-aU)}^t=0$. As $D$ is a domain, $a \neq 0$, and $U$ is a unit, we have ${(I_n-aU)}^t=0$. Hence $I_n-(I_n-aU)=aU$ is a unit. So $(aU)B=I_n$ for a $B \in \mathbb{M}_n(D)$. Therefore, $a \in U(D)$.
The converse implication is obvious by Proposition \ref{proposition1.1}.
\end{proof}

\begin{lem}\label{Lemma 1.3}
Let $R$ be a reduced ring. The full matrix ring $\mathbb{M}_n(R)$ is strongly unit nil-clean if and only if $R$ is unit-regular.
\end{lem}

\begin{proof}
Assume that $\mathbb{M}_n(R)$ is strongly unit nil-clean. Take a completely prime ideal $P$ of $R$. The natural homomorphism $R \rightarrow R/P$ induces the surjective homomorphism $\mathbb{M}_n(R) \rightarrow \mathbb{M}_n(R/P)$. Hence $\mathbb{M}_n(R/P)$ is strongly unit nil-clean. Now Lemma \ref{Lemma1.2} guarantees that $R/P$ is a division ring. In view of \cite[Theorem 1.21]{Goodearl}, $R$ is a regular ring. It follows from \cite[Corollary 4.2]{Goodearl} that $R$ is unit-regualr. The converse implication is obvious by Proposition \ref{proposition1.1}.  
\end{proof}

\begin{thm}\label{Theorem1.2}
Let $R$ be a ring.
\begin{enumerate}
\item [(1)]
If $R$ is a weakly $2$-primal strongly unit nil-clean ring then $\mathbb{M}_n(R)$ is strongly unit nil-clean.
\item [(2)]
If $R$ is an $NI$-ring, and $\mathbb{M}_n(R)$ is strongly unit nil-clean then $R$ is strongly unit nil-clean.
\end{enumerate}
\end{thm}

\begin{proof} 
(1) \quad As the ring $R$ is weakly $2$-primal, $Nil(R)=L(R)$. So $R/Nil(R)$ is a reduced unit nil-clean i.e., it is unit-regular. Now it follows from Proposition \ref{proposition1.1} that  $\mathbb{M}_n(R/Nil(R))$ is unit-regular and, consequantly it is strongly unit nil-clean. On the other hand, $\mathbb{M}_n(Nil(R))$ is a nil ideal of $\mathbb{M}_n(R)$ by \cite[Proposition 1.1]{Barati}. But the isomorphism $\frac{\mathbb{M}_n(R)}{\mathbb{M}_n(Nil(R))} \cong \mathbb{M}_n(\frac{R}{Nil(R)})$ and \cite[Lemma 2.1]{SheibaniSUN} imply that the full matrix ring $\mathbb{M}_n(R)$ is strongly unit nil-clean.

(2) \quad The quotient ring $R/Nil(R)$ is  reduced, and $\mathbb{M}_n(R/Nil(R))$ is strongly unit nil-clean as homomorphic image of $\mathbb{M}_n(R)$. But Lemma \ref{Lemma 1.3} implies that $R/Nil(R)$ is unit-regular. Now \cite[Lemma 2.1]{SheibaniSUN} gives the desired result, as respected.
\end{proof}

\begin{cor}\label{Corollary1.1}
Let $R$ be a weakly $2$-primal ring. Then $\mathbb{M}_n(R)$ is strongly unit nil-clean if and only if $R$ is strongly unit nil-clean.
\end{cor}

\begin{rem}
Corollary \ref{Corollary1.1} can be considered  generalization of \cite[Theorem 1]{Hirano}.
\end{rem}

\begin{cor}\label{Corollary1.2}
Let $R$ be a (left) right noetherian $NI$-ring. Then $M_n(R)$ is strongly unit nil-clean if and only if $R$ is strongly unit nil-clean.
\end{cor}

\begin{proof}
It follows from \cite[Theorem 10.30]{lam2013first} that $Nil(R)$ is nilpotent. So $R$ is weakly $2$-primal. Now Corollary \ref{Corollary1.1} ensures the validity of the statement. 
\end{proof}

\begin{rem}
Theorem \ref{Theorem1.2} improves  \cite[Theorem 3.8]{Sheibani} and \cite[Theorem 2.6]{SheibaniSUN}. 
\end{rem}

By utilizing \cite[Lemma 2.1]{SheibaniSUN}, one can infer that every semi-local ring with a nil Jacobson radical is strongly unit nil-clean. In the simple following proposition, we show $M_n(R)$ is strongly unit nil-clean whenever $R$ is a semi-local ring with a locally nilpotent Jacobson radical.

\begin{prop}\label{Proposition1.3}
If $R$ is a semi-local ring with a locally nilpotent Jacobson radical then the full matrix ring $\mathbb{M}_n(R)$ is strongly unit nil-clean.
\end{prop}

\begin{proof}
As $R$ is a semi-local ring, the quotient ring $R/J(R)$ is semi-simple. It is well-known $R/J(R)$ is unit-regular. Now the isomorphism $\frac{\mathbb{M}_n(R)}{\mathbb{M}_n(J(R))} \cong \mathbb{M}_n(\frac{R}{J(R)})$, \cite[Proposition 1.1]{Barati}, and \cite[Lemma 2.1]{SheibaniSUN} give the desired result.   
\end{proof}

\begin{ex}
For a right (left) perfect ring $R$, the full matrix ring $\mathbb{M}_n(R)$ is strongly unit-nil clean.
\end{ex}

Let $A$ , $B$ be two rings and $M$ , $N$ be ($A,B$)-bimodule and ($B,A$)-bimodule, respectively. Also, we consider the bimodule homomorphisms $\phi:M \otimes_B N \rightarrow A$ and $\psi: N \otimes_A M \rightarrow B$ that apply to the following properties.
\[\forall b \in N \quad b \varphi(m \otimes n)=\psi(b \otimes m)n, \quad \forall a \in M \quad a \psi(n \otimes m)=\varphi(a \otimes n)m.\]

For $m \in M$ and $n \in N$, set $mn:=\varphi(m \otimes n)$ and $nm:= \psi(n \otimes m)$. So $b(mn)=(bm)n$ for all $b \in N$ and $a(nm)=(an)m$ for all $a \in M$. The algebraic structure $(A,B,M,N,\varphi,\psi)$ is said to be Morita context. The bimodule homomorphisms $\varphi$ and $\psi$ are called pairings.

Now the 4-tuple $R=\begin{pmatrix} A & M \\ N & B \end{pmatrix}=\left\{\begin{pmatrix} a & m \\ n & b\end{pmatrix}| a \in A, m \in M, n \in N, b \in B \right\}$ becomes to an associative ring with obvious matrix operations that is called a \emph{Morita context ring}. Denote two-side ideals $Im\phi$ and $Im\psi$ to $MN$ and $NM$ respectively and that are called the \emph{trace} ideals of the Morita context. If $MN=0$ and $NM=0$, then $R$ is called a Morita context with zero pairings.

\begin{thm}\label{Theorem1.8}
Let $R=\begin{pmatrix}
A & M \\
N & B
\end{pmatrix}$ be a Morita context such that $MN$ and $NM$ are nilpotent ideals. Then $R$ is (strongly) unit nil-clean if and only if the rings $A$ and $B$ are (strongly) unit nil-clean.  
\end{thm}

\begin{proof}
 We have $R=S+K$, where 
$S=\begin{pmatrix}
A & 0 \\
0 & B
\end{pmatrix}$ is an unital subring of $R$ and 
$K=\begin{pmatrix}
MN & M \\
N & NM
\end{pmatrix}$ is a nilpotent ideal of $R$ (see the proof of \cite[Theorem 5.6]{Barati2023}). By Lemma \ref{Lemma1.4}, $R$ is (strongly) unit nil-clean if and only if $A$ and $B$ are (strongly) unit nil-clean.         
\end{proof}

The Morita context $R=\begin{pmatrix} A & M \\ 0 & B \end{pmatrix}$ 
is called a {\it formal triangular matrix ring} and denoted by $T(A,B,M)$.

\begin{cor}\label{Corollary1.7}
Let $A$, $B$ be two rings and let $M$ be an $(A,B)$-bi-module. Then the formal triangular matrix ring $T(A,B,M)$ is (strongly) unit nil-clean if and only if $A$ and $B$ are (strongly) unit nil-clean rings.
\end{cor}

The trivial extension of $A$ by an ($A,A$)-bimodule $M$ is the subring $\{\begin{pmatrix}
a & m \\
0 & a
\end{pmatrix}| a \in A , m \in M \}$ of $T(A,A,M)$ and denote it by $A \propto M$.

\begin{prop}
Let $A$ be a ring and let $M$ be a ($A, A$)-bimodule. Then the trivial extension $A \propto M$ is (strongly) unit nil-clean if and only if $A$ is (strongly) unit nil-clean.  
\end{prop}

\begin{proof}
We can consider $T=A \propto M$ as a graded ring with $T=T_0 \oplus T_1$, where $T_0=\{\begin{pmatrix}
a & 0 \\
0 & a
\end{pmatrix}| a \in R\} \cong A$ and $T_1=\begin{pmatrix}
0 & M \\
0 & 0
\end{pmatrix}$. Now Theorem \ref{Theorem1.5} gives the desired result. 
\end{proof}

Given a ring $R$ and a central element $s$ of $R$, the $4$-tuple $\begin{pmatrix} R & R \\ R & R \end{pmatrix}$ becomes a ring with addition defined componentwise and with multiplication defined by
\[\begin{pmatrix} a_1 & x_1 \\ y_1 & b_1 \end{pmatrix} \begin{pmatrix} a_2 & x_2 \\ y_2 & b_2 \end{pmatrix} =
\begin{pmatrix} a_1a_2+sx_1y_2 & a_1x_2+x_1b_2 \\ y_1a_2+b_1y_2 & sy_1x_2+b_1b_2 \end{pmatrix}. \]
This ring is denoted by $K_s(R)$. A Morita context $\begin{pmatrix} A & M \\ N & B \end{pmatrix}$ with $A=B=M=N=R$ is called a {\it generalized matrix ring} over $R$. It was observed in \cite{KrylovGMR} that a ring $S$ is a generalized matrix ring over $R$ if and only if $S=K_s(R)$ for some $s\in \mathrm{C}(R)$. Here $MN=NM=sR$, so that $MN \subseteq J(A) \Leftrightarrow s \in J(R)$, $NM \subseteq J(B) \Leftrightarrow s \in J(R)$, and $MN,NM$ are nilpotent $\Leftrightarrow s$ is a nilpotent. 
Thus, Theorem \ref{Theorem1.8} has the following consequences, too.

\begin{cor}
Let $R$ be a ring and $s \in C(R) \cap Nil(R)$. Then the generalized matrix ring $K_s(R)$ is (strongly) unit nil-clean if and only if $R$ is (strongly) unit nil-clean.
\end{cor}

For $n \geq 2$ and  $s\in C(R)$, the $n \times n$ formal matrix ring over $R$, associated with $s$ denoted by $\mathbb{M}_n(R;s)$, can be defined same as in \cite{TangZhou}. Now, we have the following result, the proof of which is fully similar to that of \cite[Theorem 2.18]{Barati2023}.

\begin{cor}
Let $R$ be a ring and $s \in C(R) \cap Nil(R)$. Then $\mathbb{M}_n(R;s)$, the formal matrix ring associated with $s$,  is (strongly) unit nil-clean if and only if the ring $R$ is (strongly) unit nil-clean.
\end{cor}

\section{Group Rings}
In the following theorem, we answer to \textbf{Question 1} with the least possible restriction.

\begin{thm}\label{Theorem2.2}
Let $R$ be a ring and $G$ be a group. If $RG$ is unit-regular then $R$ is a unit-regular ring, $G$ is a locally finite group and the order of every element of $G$ is a unit in $R$. The converse is valid when $R$ is a ring with the noetherian prime factors.  
\end{thm}

\begin{proof}
(1) $\Rightarrow$ (2) \quad 
Being a homomorphic image of $RG$, $R$ is unit-regular. On the other hand, $G$ is a locally finite group and the order of any finite subgroup of $G$ is a unit in $R$ by \cite[Theorem 3]{Connell}. For $g \in G$, put $H=<g>$. Then $H$ is finite and $|H|=o(g)$ is a unit in $R$.

(2) $\Rightarrow$ (1) \quad
Suppose $H$ is a finite subgroup of $G$. Assuming $|H|={(p_1)}^{\alpha_1} \cdots {(p_k)}^{\alpha_k}$, it follows from Cauchy's Theorem that $H$ contains the elements such as $g_1$, $g_2$, $\cdots$ , $g_k$ of orders $p_1$, $p_2$ $\cdots$ $p_k$, respectively. So $|H|={(o(g_1))}^{\alpha_1} \cdots {(o(g_k))}^{\alpha_k}$. As $o(g_i)$ is reversible in $R$, $|H|$ is so. Assume $P$ is an arbitrary prime ideal of $R$. As $R/P$ is a unit-regular noetherian ring, it is artinian by \cite[Corollary 4.26]{lam2013first}. Therefore, $RG$ is strongly $\pi$-regular ring by \cite[Theorem 3.3]{ChinChen} and it is regular by Theorem \ref{Theorem0.1}. Now \cite[Theorem 5.8]{GoodearlMenal} implies that $RG$ is unit regular. 
\end{proof}

\begin{cor}\label{Corollary2.1}
Suppose $R$ is a ring with the one of the following properties. 

a. $Nil(R)$ is an ideal or a commutative set. 

b. semi-abelian

c. semilocal

 then the following statements are equivalent: 
\begin{enumerate}
\item [(1)]
$RG$ is  unit-regular ring.
\item [(2)]
$R$ is a unit-regular and $G$ is a locally finite group whose the order of every element is a unit in $R$.
\end{enumerate} 
\end{cor}

\begin{proof}
Assume $R$ is unit-regular. If $R$ is an $NI$-ring then $Nil(R)=0$. If $R$ is semi-abelian $Nil(R)$  is an ideal by \cite[Theorem 3.11]{ChenW}. So $Nil(R)=J(R)=0$. If $Nil(R)$ is commutative then \cite[Theorem 2.8]{KhuranaMarks} ensures $R$ is reduced. However, $R$ is reduced. 
In a similar way to the argument presented in Theorem \ref{Theorem2.1}, it can be demonstrate that for any prime ideal $P$ of $R$, $R/P$ is a division ring. If $R$ is a semilocal unit-regular ring then it is semisimple. Consequently, the quotient ring $R/P$ is artinian for every prime ideal $P$ of $R$. Now Theorem \ref{Theorem2.2} completes the argument.
\end{proof}

\begin{cor}[{\cite[Corollary 2.5]{ChenLiZhou}}]\label{Corollary2.2}
If $R$ is a semisimple ring, then the following statements are equivalent:
\begin{enumerate}
\item [(1)]
$RG$ is unit regular.
\item [(2)]
$RG$ is regular
\item [(3)]
$G$ is locally finite and each $n \in o(G)$ is a unit in $R$, where $o(G)$ denotes the set of orders
of all finite subgroups of $G$.
\end{enumerate}
\end{cor}

An element $a$ in a ring $R$ is called left morphic if there exists $b \in R$ such that $l_R(a) = Rb$ and
$l_R(b) = Ra$, where $l_R(a)$ denotes the left annihilator of $a$ in $R$. The ring $R$ is called left morphic if
every element of $R$ is left morphic. Analogously, one can define right morphic rings. A left and right morphic ring is called a morphic
ring. 
It follows from a well-known theorem of Erlich \cite{Ehrlich} that an element $a$ in R is unit regular if and only if $a$ is (von Neumann) regular and left morphic. Also, It was proved in \cite{NicholsonSanchez} that $\mathbb{Z}_n$ is morphic. Therefore, being regularity $\mathbb{Z}_n$ is equivalent to its unit-regularity. According to this, we will obtain the following simple and key lemma.

\begin{lem}\label{Lemma2.1}
$\mathbb{Z}_n$ is unit-regular if and only if $n$ is a product of the distinct primes. 
\end{lem}

\begin{proof}
To prove the statement, we first show that $\mathbb{Z}_{p^r}$ is unit-regular if and only if $r=1$, where $p$ is a prime. If $r=1$ then $\mathbb{Z}_p$ is, as a field, unit-regular. Assume $\mathbb{Z}_{p^r}$ is unit-regular. It follows from \cite[Theorem 3]{Alkam} that $p^r=p^r-p^{r-1}+1$. Thus $r=1$. Now suppose $n={p_1}^{r_1}{p_2}^{r_2}\cdots {p_k}^{r_k}$. Then $\mathbb{Z}_n\cong \mathbb{Z}_{{p_1}^{r_r}} \times \cdots \mathbb{Z}_{{p_k}^{r_k}}$. So, $\mathbb{Z}_n$ is unit-regular if and only if $\mathbb{Z}_{{p_i}^{r_i}}$ is unit-regular if and only if $r_i=1$. 
\end{proof}

\begin{thm}\label{Theorem2.3}
The group ring $\mathbb{Z}_n G$ is unit-regular if and only if $n$ is a product of the distinct primes and $G$ is a locally finite group whose the order of every element is relatively prime to $n$.
\end{thm}

\begin{proof}
By Corollary \ref{Corollary2.1}, $\mathbb{Z}_n G$ is unit-regular if and only if $\mathbb{Z}_n$ is unit-regular and $G$ is a locally finite group whose the order of every element is a unit in $\mathbb{Z}_n$. Lemma \ref{Lemma2.1} implies that $\mathbb{Z}_n$ is unit-regular if and only if $n$ is a product of the distinct primes. Also, the order of $g \in G$ is a unit in $\mathbb{Z}_n$ if and only if $(n,o(g))=1$. 
\end{proof}

\begin{cor}
For the finite group $G$, the group ring $\mathbb{Z}_nG$ is unit-regular if and only if $n$ is a product of the distinct primes and $(n,|G|)=1$.
\end{cor}

\begin{prop}
For a prime $p$, let $R$ be a (strongly) unit nil-clean ring with $p \in Nil(R)$ and $G$ be a locally finite $p$-group. Then the group ring $RG$ is (strongly) unit nil-clean.
\end{prop}

\begin{proof}
By \cite[Theorem 9]{Connell}, the augmentation ideal $\Delta(RG)$ is a nilpotent ideal of $RG$. It follow from the isomorphism $\frac{RG}{\Delta(RG)} \cong R$ and \cite[Lemma 3.1]{Sheibani} (\cite[Lemma 2.1]{SheibaniSUN}) that $RG$ is (strongly) unit nil-clean.
\end{proof}

\noindent In the sequel, we try to give a partial answer to \textbf{Question 2}.

\begin{thm}\label{Theorem2.1}
Let $R$ be a ring and $G$ be a group.
\begin{enumerate}
\item [(1)] 
If $R$ is a weakly $2$-primal strongly unit nil-clean and $G$ is a locally finite group then the group ring $RG$ is strongly unit nil-clean.
\item [(2)]
If $RG$ is a strongly unit nil-clean ring then $R$ is strongly unit nil-clean and $G$ is a torsion group. 
\end{enumerate} 
\end{thm}

\begin{proof}
(1) \quad By the hypothesis, $Nil(R)=L(R)$. As the quotient ring $\overline{R}=\frac{R}{Nil(R)}$ is a reduced strongly unit nil-clean, it is unit-regular. Let $\overline{Q}$ be any minimal prime ideal of $\overline{R}$. According to Lemma 1.20 in \cite{Goodearl}, $\overline{Q}$ is completely prime. Since $\overline{R}/\overline{Q}$ is a domain, it follows that $\overline{R}/ \overline{Q}$ is a division ring. Consequently, we see that $\overline{R}/ \overline{P}$ is a division ring for every prime ideal $\overline{P}$ of $\overline{R}$. On the other hand, \cite[Theorem 3.3]{ChinChen} ensures that the group ring $\overline{R}G$ is strongly $\pi$-regular and consequently it is strongly unit nil-clean by \cite[Corollary 2.3]{SheibaniSUN}. We also know the ideal $Nil(R)G$ is nil. Finally, the isomorphism $\frac{RG}{Nil(R)G} \cong (\frac{R}{Nil(R)})G$ and \cite[Lemma 2.1]{SheibaniSUN} show that $RG$ is strongly unit nil-clean.

(2) \quad Being a homomorphic image of $RG$, $R$ is strongly unit nil-clean.  Take $1 \neq g \in G$. According to \cite[Theorem 2.2]{SheibaniSUN}, for $x=1-g \in RG$, there are a unit $u \in RG$ and a natural number $t>0$ such that ${(ux-{(ux)}^2)}^t=0$. If $t=1$ then it is easy to see that $x$ is a zero-divisor. Let $t>1$. So $x\underbrace{[{(ux)}^{t-1}{(1-ux)}^t]}_{*}=0$. If (*) is nonzero then $x$ is a zero-divisor. Suppose (*) equal to zero. Thus $x\underbrace{[{(ux)}^{t-2}{(1-ux)}^t]}_{**}=0$. Again, if (**) is nonzero then $x$ is a zero-divisor otherwise, suppose $\beta$ is the smallest natural number such that ${(ux)}^{\beta}{(1-ux)}^t=0$. If $\beta=1$ then $x{(1-ux)}^t=0$. It is clear that ${(1-ux)}^t \neq 0$ otherwise, $ux=1-(1-ux)$ is a unit, that is a contradiction. So $x$ is a zero-divisor in this case, too. If $\beta>1$ then
$x\underbrace{[{(ux)}^{\beta-1}{(1-ux)}^t]}_{***}=0$. By the hypothesis, (***) is nonzero. However, $x=1-g$ is a zero-divisor. Now \cite[Proposition 6]{Connell} yields that $g$ is of finite order.
\end{proof}

\begin{cor}[{\cite[Proposition 3.4]{ChinChen}}]
Let $R$ be a ring and $G$ a group. If $RG$ is strongly $\pi$-regular, then $R$
is strongly $\pi$-regular and $G$ is torsion.
\end{cor}

A group $G$ is called locally nilpotent if every finitely generated
subgroup of $G$ is nilpotent. It is obvious that an abelian group and a nilpotent group are locally nilpotent.

\begin{cor}
Let $R$ be a weakly $2$-primal ring and $G$ be a locally nilpotent group. Then the group ring $RG$ is strongly unit nil-clean if and only if $R$ is an strongly unit nil-clean ring and $G$ is a locally finite group.
\end{cor}

\begin{proof}
It follows from \cite[Theorem 2.25]{ClementMZ} and Theorem \ref{Theorem2.1}. 
\end{proof}

\begin{cor}
Let $R$ be a commutative ring and $G$ be an abelian group. Then $RG$ is strongly unit nil-clean if and only if $R$ is strongly unit nil-clean and $G$ is a torsion group.
\end{cor}

The following corollary gives a partial answer to the \textbf{Queation 2}

\begin{cor}
If $R$ is a unit-regular $NI$-ring and $G$ is a locally finite group then the group ring $RG$ is strongly unit nil-clean. 
\end{cor}

\begin{prop}\label{Proposition2.3}
Let $R$ be a periodic (left) right noetherian ring and $G$ be a locally finite group. Then the group ring $RG$ is periodic if and only if $R$ is periodic.
\end{prop}

\begin{proof}
Take an element $x=\sum_{g \in G} r_gg \in RG$. If $H_x$ is a subgroup of $G$ generated by the support of $x$ then it is clear $x \in RH_x$. As $G$ is locally finite, the group $H_x$ is finite. Set $n=|H_x|$. By \cite[Lemma 1]{Woods}, $RH_x$ is a subring of $\mathbb{M}_n(R)$. But Theorem 2.19 in \cite{Bouzidi} ensures that $\mathbb{M}_n(R)$ is periodic. Therefore, $RH_x$ and consequently $RG$ is periodic. 
\end{proof}

For a ring $R$ and a fixed natural number $m>1$, the element $e \in R$ is called $m$-potent when $e^m=e$. The ring $R$ is strongly $m$-nil clean if each of its element is the sum of an $m$-potent and a nilpotent which commute with each other. Now compare the following result with \cite[Theorem 5.4(1)]{BMA} or even with \cite[Theorem 4.7(2)]{KWZ}.
 
\begin{ex}
For an integer $m>1$, if $R$ is a strongly $m$-nil clean (left) right noetherian ring and $G$ is a locally finite group then the group ring $RG$ is periodic. 
\end{ex}

The Proposition \ref{Proposition2.3} inspires us to pose this question. If $R$ is a (left) right noetherian (strongly) unit nil-clean ring and $G$ is a locally finite group then is the group ring $RG$ (strongly) unit nil-clean or not? In the sequel, we answer, with an additional restriction,  to this question.

\begin{cor}
Let $R$ be a (left) right noetherian $NI$-ring. If $R$ is a unit nil-clean ring and $G$ is a locally finite group then the group ring $RG$ is strongly unit nil clean. 
\end{cor}

\begin{proof}
By \cite[Theorem 10.30]{lam2013first}, $Nil(R)$ is nilpotent. Hence  $R$ is a weakly $2$-primal unit nil-clean ring. Now Theorem \ref{Theorem2.1} completes the argument.
\end{proof}

\begin{prop}\label{Proposition2.2}
Suppose $R$ is a semi-local ring with a locally nilpotent Jacobson radical and $G$ is a locally finite group. Then the group ring $RG$ is strongly unit nil-clean. 
\end{prop}

\begin{proof}
As $R$ is a semi-local ring, $R/J(R)$ is semi-simple. Consider $x=\sum_{g \in G}^{} r_gg \in (R/J(R))G$. Let $H_x$ be the subgroup of $G$ generated by the support of $x$. It is clear $x \in (R/J(R))H_x$. As $G$ is locally finite, $H_x$ is finite. Being semisimple $R/J(R)$ implies that it is Artinian. So $(R/J(R))H_x$ is Artinian; hence strongly $\pi$-regular. It follows that $(R/J(R))G$ is strongly $\pi$-regular, and consequantly it is strongly unit nil-clean by \cite[Corollary 2.3]{SheibaniSUN}. On the other hand, similar to the proof of \cite[Theorem 1.2]{Barati2023}, the ideal $J(R)G$ is nil. Now the isomorphism $RG/J(R)G \cong (R/J(R))G$ and \cite[Lemma 2.1]{SheibaniSUN} yield that $RG$ is a strongly unit nil-clean ring.  
\end{proof}

\begin{cor}\label{Corollary2.3}
If $R$ is a right (left) perfect ring and $G$ is a locally finite group, then the group ring $RG$ is strongly unit nil-clean.
\end{cor}

In the following example, we give a particular case of Corollary \ref{Corollary2.3}.
\begin{ex}\label{Example2.1}
If R is a semisimple ring and $G$ is a locally finite group then $RG$ is strongly unit nil-clean.
\end{ex}

\begin{rem}
Corollary \ref{Corollary2.3} is indeed the other answer to the \textbf{Question 2}. 
\end{rem}

\bibliographystyle{plain}
\bibliography{Paper4}
\end{document}